\documentclass{amsart}

\usepackage[latin1]{inputenc}
\usepackage{amsfonts}
\usepackage{amsmath}
\usepackage{amsthm}
\usepackage{amssymb}
\usepackage{latexsym}
\usepackage{enumerate}

\newtheorem{theorem}{Theorem}[section]
\newtheorem{lemma}[theorem]{Lemma}
\newtheorem{proposition}[theorem]{Proposition}

\theoremstyle{definition}
\newtheorem{definition}[theorem]{Definition}

\numberwithin{equation}{section}

\begin{document}

\newcommand{\cc}{\mathfrak{c}}
\newcommand{\N}{\mathbb{N}}
\newcommand{\Q}{\mathbb{Q}}
\newcommand{\R}{\mathbb{R}}
\newcommand{\PP}{\mathbb{P}}
\newcommand{\forces}{\Vdash}
\newcommand{\st}{*}

\title [Equilateral sets] 
{Uncountable equilateral sets in Banach spaces of the form $C(K)$}
\author{Piotr Koszmider}
\address{Institute of Mathematics, Polish Academy of Sciences,
ul. \'Sniadeckich 8,  00-656 Warszawa, Poland}
\email{\texttt{piotr.koszmider@impan.pl}}
\thanks{The  research was  supported by   grant
PVE Ci\^encia sem Fronteiras - CNPq, process number 406239/2013-4 }


\begin{abstract} 
The paper is concerned with the problem whether a nonseparable
Banach space must contain an uncountable 
  set of vectors such that the distances between every two distinct vectors of the set are the same.
Such sets are called equilateral.
We show that Martin's axiom and the negation of the continuum hypothesis
imply that  every nonseparable Banach space of the form $C(K)$ has
an uncountable equilateral set. We also show that one cannot obtain such a result without
an additional
set-theoretic assumption since we construct 
an
example of 
 nonseparable Banach space of the form $C(K)$
which has no uncountable equilateral set (or equivalently 
no uncountable $(1+\varepsilon)$-separated set in the unit sphere
for any $\varepsilon>0$) making another consistent 
combinatorial assumption.  The compact $K$ is a version of the split interval
obtained from a sequence of functions which behave in an anti-Ramsey manner.
  It remains open if there is an absolute example of  a nonseparable Banach space
of the form different than $C(K)$ which  has no uncountable equilateral set.
It follows from the results of S. Mercourakis, G. Vassiliadis
that our example has an equivalent renorming in which it has an uncountable
equilateral set. It remains open if there
are consistent examples which have  no uncountable equilateral sets
in any equivalent renorming but it follows from the results of S. Todorcevic that it is consistent
that every nonseparable Banach space has an equivalent renorming 
in which it has an uncountable
equilateral set.  
\end{abstract}

\maketitle

\markright{}


\section{Introduction}

In this note we deal with the possible behaviour of the norm on uncountable
sets of vectors in nonseparable Banach spaces. Our examples and results refer to  Banach
spaces of the form $C(K)$ that is spaces of real-valued continuous functions on a compact Hausdorff space $K$
with the supremum norm.
However, our example is a novelty also in the entire class of Banach spaces.
Recall that for $r\in \R$ a set $A$ of a Banach space $X$ is $r$-separated if $||v_1-v_2||> r$ for
distinct $v_1, v_2\in A$ and it is $r$-equilateral if $||v_1-v_2||=r$
for all  distinct $v_1, v_2\in A$. A set is called equilateral if it is $r$-equilateral
for some $r\in \R$. 

A natural question considered in the literature is
whether given a Banach space of big  dimension there is a big equilateral set or 
there is a big $1$-separated or $(1+\varepsilon)$-separated set in the unit sphere of
the Banach space. The classical Riesz lemma gives the optimal answer for
$(1-\varepsilon)$-separation in the  finite dimensional case.
In the infinite dimensional separable case Kottman showed in
\cite{kottman} the existence of an infinite $1$-separated set and Elton with Odell
improved it for $(1+\varepsilon)$-separation for some $\varepsilon>0$.
However there are infinite dimensional Banach spaces without 
infinite equilateral sets as first shown by Terenzi (\cite{terenzi}, see also \cite{noequilateral}).

Already Elton and Odell noted that $c_0(\Gamma)$ for uncountable $\Gamma$
has no uncountable $(1+\varepsilon)$-separated sets for some $\varepsilon$. 
It is clear however that it has an uncountable equilateral set. 
Our main result is to give
the first  example of a nonseparable Banach space with no uncountable
equilateral set. This example is not absolute in the sense that the construction
requires an additional set-theoretic assumption consistent with the usual axioms ZFC.
The example is of the form $C(K)$, and we are able to show that
in this class of Banach spaces an additional set-theoretic assumption is necessary.
This undecidability is compatible with many results of this type which keep
appearing in recent years in all areas of
classical mathematics from operator algebras through homology theory to 
the elementary properties of the Lebesgue measure and integral,
where the continuum  hypothesis provides pathological examples while 
a certain hierarchy of  axioms which starts with Martin's axiom with the negation
of the continuum hypothesis provide
nice structural theory based on canonical examples, selection principles, dichotomies etc.

Let us discuss now the content of the paper in more details. In Section 2 we 
define (Definition \ref{splitdefinition}) a class of compact nonmetrizable spaces called connectedly split intervals
which are obtained by splitting points of a subset $\{r_\xi: \xi<\kappa\}$ of
the  interval $[0,1]$. These are connected generalizations
of the classical split interval (e.g., \cite{godefroy} and
for totally disconnected similar generalizations see also \cite{rolewicz} and \cite{biorthogonal}). 
We analyze their relatively
simple structure by embedding them in the products of intervals.
They depend of the choice of the splitting functions 
$f_\xi: [0,1]\setminus\{r_\xi\}\rightarrow[-1,1]$ which are continuous
in their domains but normally do not extend to a continuous function
on $[0,1]$ and which determine the way the point $r_\xi$ is split in the final space.
For our examples the prototypical splitting function is $\sin({1\over r-r_\xi})$.
Section 2 contains also simple properties of such spaces (Proposition \ref{splitproperties}) and in particular
a lemma (Lemma \ref{representation}) about simple representation of elements of a norm
dense subset of $C(K)$ for
$K$ in our class. It is remarkable that all continuous functions can be approximated by
very simplified combinations of functions generating the Banach algebra $C(K)$.

In Section 3, in Definition \ref{randomdef} we introduce 
anti-Ramsey sequences $(f_\xi)_{\xi<\omega_1}$
of splitting functions. 
Similar constructions were used in our papers \cite{rolewicz} and \cite{biorthogonal}
where we constructed Banach spaces without uncountable semibiorthogonal sequences or
without specific uncountable biorthogonal sequences. However, for simplicity the present
construction contains simple uncountable biorthogonal systems.  This with
Corollary 1 of \cite{equilateral} shows that our space $C(K)$ has an equivalent renorming
in which it has an uncountable equilateral system.
This also shows that the presence of such systems does not imply the existence
of uncountable equilateral sets. We do not know if stronger consistent versions could be 
obtained where no renorming has uncountable equilateral set.

Using the representation lemma (\ref{representation})
we prove one of our main results (Theorem \ref{maintheorem}) that if $(f_\xi)_{\xi<\omega_1}$ is anti-Ramsey then
the obtained $C(K)$ has no uncountable equilateral set. It follows from
the work of  Mercourakis and Vassiliadis in \cite{equilateralcofk} 
that such spaces  cannot have in their unit sphere an uncountable
$(1+\varepsilon)$-separated set for any $\varepsilon>0$. On the
other hand Kania and Kochanek
proved  in \cite{kania} that every nonseparable $C(K)$ space has an uncountable $1$-separated set.
Roughly speaking in our $C(K)$ spaces (like in the spaces of \cite{rolewicz, biorthogonal}) 
among any uncountable set of vectors there are two vectors which are related as we wish,
of course up to some limitations which are built in the definition 
of a connectedly split interval. This type of structures was first systematically considered
probably in \cite{shelah} by S. Shelah. Since then a remarkable progress was made,
one of the biggest steps of which were achieved by removing special set-theoretic assumptions
from some of such constructions using
the anti-Ramsey coloring due to Todorcevic of
\cite{rho}. It was applied in the context
of Banach spaces  first in  the work \cite{shelahsteprans} of Shelah and Steprans
and later \cite{wark} of Wark (further
advanced in \cite{argyros} or in \cite{distortion}). We believe
that it  is a natural tool (see also \cite{walks}) which
one can try to use to construct an absolute example of a nonseparable Banach
space without an uncountable equilateral set. 
However,
a result of Todorcevic from \cite{stevobiorth} says that it is consistent that every
nonseparable Banach space contains an uncountable biorthogonal system which
gives by Corollary 1 of \cite{equilateral} a renorming in which it has an uncountable equilateral set.
One could  also mention  here related results of \cite{spinas} concerning general vector spaces.

It is  immediate that if $K$ is totally disconnected, then $C(K)$ has a $2$-equilateral 
set of the maximal possible cardinality i.e., of the density of $C(K)$. For this one just
takes $\chi_A-\chi_{K\setminus A}$ when $A$ runs through all  clopen subsets of the compact $K$.
This shows that our $K$ has no nonmetrizable
totally disconnected subspace. Such spaces were known to exist under special set-theoretic
assumptions. It remains open if such spaces can be constructed
without any extra assumption or  it is consistent that every
compact nonmetrizable $K$ has a nonmetrizable totally disconnected subspace (see \cite{opit}).
It is in the proof of Proposition \ref{splitproperties} that we use the Darboux property
of continuous functions on connected spaces based on which we can prove
later that our Banach space does not contain an uncountable equilateral set.
This argument would not go through for our totally disconnected versions
of the split interval considered in the papers \cite{biorthogonal, rolewicz}.

We also do not know
if one can consistently construct a nonseparable space $C(K)$ where all renormings
do not have uncountable equilateral sets. Such examples would be new examples
of spaces nonisomorphic to $C(L)$s for totally disconnected $L$.
The first such examples were obtained only recently  in \cite{few} with $K$ very big (e.g.
strongly rigid, see \cite{iryna}) and later 
with $K$ sequentially compact in \cite{duke}, but the spaces like in this paper
would lead to   first countable examples. Recall that it is not even known if the ball
in the nonseparable Hilbert space with the weak topology is such a compact space.

In Section 4 we show the consistency of the existence of 
anti-Ramsey sequence of splitting functions. This is done by a forcing argument
which follows the lines of \cite{rolewicz} and \cite{biorthogonal}.
This section can be skipped by the readers not familiar with forcing.
The readers which are forcing curious but do not want to
go through a systematic forcing course as in \cite{kunen} may
consult books like \cite{ciesielski} or \cite{weaver}. Our forcing argument
belongs to the simplest possible, it is a single forcing and we are interested in 
elementary properties of the generic object. This argument shows
that our example is consistent both with the continuum hypothesis and its negation.
It seem an interesting technical question if such an example can be obtained 
just from the continuum hypothesis, for example along the lines of \cite{notre}.

Section 5 is devoted to proving that assuming
 Martin's axiom and the  negation of the continuum hypothesis every nonseparable Banach space of the form
$C(K)$ contains an uncountable equilateral set.
This is based on the use of an  uncountable  intersecting family of closed pairs introduced in
\cite{equilateralcofk} which is equivalent
by Theorem 1 of \cite{equilateralcofk} to the existence
of $2$-equilateral uncountable set in the sphere of $C(K)$ and to the existence of a $(1+\varepsilon)$-separated
set for some $\varepsilon>0$. It turns out that such an uncountable
intersecting family is an antichain in a forcing notion used extensively e.g. in \cite{formin},
so if it does not exist, some natural forcing notions are c.c.c. which allows us
to use Martin's axiom. It should be added that it was known (see \cite{stevobiorth})
that Martin's axiom implies the existence of uncountable biorthogonal systems
in nonseparable Banach spaces of the form $C(K)$ and hence an uncountable equilateral
set in some renorming of the $C(K)$ by \cite{equilateral}.

The notation and terminology is relatively standard, concerning Banach
spaces we follow the book \cite{fabian} and concerning  set theory \cite{kunen}.
The Banach spaces we consider are of the form $C(K)$ for $K$ compact, infinite and Hausdorff
with the supremum norm $\|\ \ \|$. Sometimes we will also consider bounded discontinuous $g: K\rightarrow \R$,
then we use the notation $\| g \|_\infty=\sup\{|g(x)|: x\in K\}$.
When a metric $d$ is fixed on some set $M$, then $diam(M)=\sup\{d(x,y): x\in M\}$.
The author would like to thank Marek Cuth for careful reading of the preliminary version
of this paper and spotting some inaccuracies.

\section{Connectedly split intervals}

Given a cardinal $\kappa$ and  functions $f_\xi: [0,1]\setminus\{r_\xi\}\rightarrow[-1,1]$
 for some distinct  $r_\xi\in [0,1]$
for $\xi<\kappa$   we may define a natural connected version of the 
split interval (see e.g. \cite{godefroy}) which can be naturally embedded in the product 
space $[0,1]\times[-1,1]^\kappa$. 
A classical prototype    of  such an $f_\xi$  is the Warsaw $\sin$ function i.e.,
$\sin({1\over{r-r_0}})$ for some $r_0\in (0,1)$.

\begin{definition}\label{splitdefinition} Let $\kappa\leq 2^\omega$  be a cardinal.
Suppose that $\{r_\xi: \xi<\kappa\}$ consists of distinct elements of  $[0,1]$ and
$f_\xi: [0,1]\setminus\{r_\xi\}\rightarrow [-1,1]$ is
 a continuous function for every $\xi<\kappa$.  A connectedly split interval
 induced by $(f_\xi)_{\xi<\kappa}$ 
is the subspace $K$ of $[0,1]^{\{\st\}}\times [-1,1]^{\kappa}$ consisting of points of the form 
\[\{x_{\xi, t}, : \xi<\kappa, t\in [-1,1]\}\cup\{x_r: r\in [0,1]\setminus\{r_\xi: \xi<\kappa\}\}, \]
where
\begin{enumerate}
\item $x_{\xi, t}(\st)=r_\xi$, $x_{\xi, t}(\xi)=t$ and $x_{\xi, t}(\eta)=f_\eta(r_\xi)$
 if $\eta\in \kappa\setminus\{\xi\}$,
\item $x_r(\st)=r$ and $x_r(\xi)=f_\xi(r)$ for all 
$r\in [0,1]\setminus\{r_\xi: \xi<\kappa\}$ and $\xi<\kappa$.
\end{enumerate}
Under these assumptions we will use the following notation and terminology:
\begin{enumerate}
\item $\overline{f_{\st}}:K\rightarrow [0,1]$ is given by $\overline{f_{\st}}(x)=x(\st)$ for any $x\in K$,
\item $\overline{f_\xi}: K\rightarrow [0,1]$ is given by $\overline{f_\xi}(x)=x(\xi)$ for all $\xi<\kappa$
and $x\in K$. 
\item $(a,b)_K=\{x\in K: a<x(\st)<b\}$, and similar meaning for $(a,b]_K$, $[a,b)_K$, $[a,b]_K$ for
all $0\leq a<b\leq 1$.
\item $(a,b)_{K,\xi}=\{x\in K: a<x(\xi)<b\}$, and similar meaning for $(a,b]_{K,\xi}$, $[a,b)_{K,\xi}$, $[a,b]_{K,\xi}$ for
all $0\leq a<b\leq 1$ and every $ \xi<\kappa$.
\item $R_\xi=\{ [x_{\xi,t}]: t\in [-1,1]\}$ for all $\xi<\kappa$,
\item $f_\xi$s will be called splitting functions.
\end{enumerate}
The coordinate in the product $[0,1]^{\{\st\}\cup\kappa}$ corresponding to $\st$
will be called the $\st$-coordinate. 
\end{definition}

Thus, the classical split interval  $S$ is obtained by choosing 
$(0,1)=\{r_\xi: \xi<2^\omega\}$, $f_\xi: [0,1]\setminus\{r_\xi\}$ defined by
$f_\xi(r)=0$ if $r<r_\xi$ and $f_\xi(r)=1$ if $r>r_\xi$ and considering
the subspace $S=K\cap([0,1]^{\{\st\}}\times\{0,1\}^{2^\omega})$ 
of the connectedly split interval $K$ induced by $(f_\xi)_{\xi<2^\omega}$.
Note that it follows from the definition of the connectedly split interval $K$ that
the only point $x$ of $K$ such that $x(\st)=r$ is $x_r$ if 
$r\in [0,1]\setminus\{r_\xi: \xi<\kappa\}$ and that the only points
 $x$ of $K$ such that $x(\st)=r_\xi$ for $\xi<\kappa$ are the points of $R_\xi$ that is  $x_{\xi, t}$s 
for $t\in [-1,1]$ and these points differ just at the $\xi$-th coordinate and are equal 
on all other coordinates of the product.

\begin{lemma}\label{connected} Suppose that $K$ is Hausdorff compact and connected, $x_0\in K$ and $f: K\setminus\{x_0\}\rightarrow
[-1,1]$ is continuous. Then
\[L=\{(x,t)\in K\times [-1,1]: x\not=x_0\ \Rightarrow \ t=f(x)\}\]
is compact and connected. 
\end{lemma}
\begin{proof} If $(x,t)\in (K\times[-1,1])\setminus L$, then $x\not=x_0$ and there is $\varepsilon>0$
such that $|f(x)-t|>\varepsilon$.  By the continuity of $f$ at $x$
 there is a neighbourhood $U$ of $x$ such that
$x_0\not\in U$ and $|f(x_1)-f(x_0)|<\varepsilon/2$ for all $x_1\in U$. It follows that
$U\times(t-\varepsilon/2, t+\varepsilon/2)$ is disjoint from $L$ and contains $(x,t)$,
hence the complement of $L$ is open and so $L$ is compact.

For connectedness suppose that $L\subseteq A\cup B$ where $A, B\subseteq K\times [-1,1]$
are open with disjoint closures and $L\cap A\not=\emptyset\not=L\cap B$. 
Define $A'\subseteq K$ to be the set of $x\in K\setminus\{x_0\}$ such that $(x, f(x))\in A$
and  $B'\subseteq K$ to be the set of $x\in K\setminus\{x_0\}$ such that $(x, f(x))\in B$. 
By the continuity of $f$ the sets
$B'$ and $A'$ are disjoint open subsets of $K$ such that $A'\cup B'=K\setminus\{x_0\}$.
We may assume that 
$(\{x_0\}\times[-1,1])\cap A\not=\emptyset$, and so by the connectedness of $[-1,1]$
we have $\{x_0\}\times[-1,1]\subseteq A$ and by the compactness of $[-1,1]$ we have
a neighbourhood $U$ of $x_0$ such that $U\times [-1,1]\subseteq A$. Thus
$A'\cup\{x_0\}$ contains $U$ and so is open and hence the partition of $K$ into $A'\cup\{x_0\}$ and $B'$
 contradicts the connectedness of $K$. 
\end{proof}

\begin{proposition}\label{splitproperties} Suppose that $\{r_\xi:\xi<\kappa\}\subseteq [0,1]$ where $\kappa$
is an uncountable cardinal. Suppose that 
$K\subseteq [0,1]^{\{\st\}}\times[-1,1]^{\kappa}$ is a generalized split interval induced by 
$(f_\xi)_{\xi<\kappa}$ where $f_\xi: [0,1]\setminus\{r_\xi\}\rightarrow [-1,1]$. Then
\begin{enumerate} 
\item $K$ is a compact Hausdorff space of weight $\kappa$,
\item $K$ is connected,
\item $K$ is first countable, 
\[\{(r-1/n: r+1/n)_K: n\in \N\}\]
 forms a basis at $x_r$ for each $r\in [0,1]\setminus\{r_\xi:\xi<\kappa\}$,
and 
\[\{(r_\xi-1/n: r_\xi+1/n)_K\cap (t-1/n, t+1/n)_{K,\xi} : n\in \N\}\]
 forms a basis at $x_{\xi, t}$ 
for each $t\in [-1,1]$ and each $\xi<\kappa$,
\item $K^2$ has a discrete set of cardinality $\kappa$,
\item $C(K)$ has a biorthogonal system of cardinality $\kappa$.
\item $\overline {f_\xi}\restriction (K\setminus R_\xi)=f_\xi\circ \overline{f_{\st}}\restriction (K\setminus R_\xi)$
\end{enumerate}
\end{proposition}
\begin{proof} (1) If $x\in [0,1]\times[-1,1]^\kappa$ is not in $K$
it is because either $x(\st)=r_\xi$ for some $\xi<\kappa$ but $f_\eta(r_\xi)\not=x(\eta)$
for some $\xi\not=\eta<\kappa$ or $x(\st)\not\in \{r_\xi:\xi<\kappa\}$ and 
$f_\eta(x(\st))\not=x(\eta)$
for some $\eta<\kappa$. In both cases one sees that these conditions define open sets
in $[0,1]^{\{\st\}}\times[-1,1]^\kappa$.

(2) We can see $K$ as the inverse limit of $K_\eta$ 
for $\eta<\kappa$ where  
$K_\eta\subseteq [0,1]^{\{\st\}}\times [-1,1]^{\eta}$ is a connectedly split interval induced by 
$(f_\xi)_{\xi<\eta}$. As inverse limits of compact connected spaces are compact connected,
it is enough to prove  that $K_{\eta+1}$ is connected provided $K_\eta$ is. This 
 follows from Lemma \ref{connected} since by the definition of the connectedly split interval $K_\eta$
the point $x_{\eta, t}\restriction\eta=r_\eta^\frown f_\xi(r_\eta)_{\xi<\eta}$ (for any $t\in [-1,1]$)
 is the unique point of $K_\eta$ such that its 
projection on the $\st$-coordinate is $r_\eta$, so $K_{\eta+1}$ is obtained from $K_\eta$
as in Lemma \ref{connected} for $f: K_\eta\setminus x_{\eta,0}\restriction\eta\rightarrow [-1,1]$
defined by $f(x) =f_\eta(x(\st))$ for any $x$ in the domain of $f$.

(3) The mentioned collections of neighbourhoods define decreasing pseudobases at indicated points,
but such pseudobases are bases in compact spaces.

(4) Consider $(x_{\xi,-1}, x_{\xi, 1})\in K\times K$ and open
$U_\xi\times V_\xi\subseteq K\times K$ given by $U_\xi=\{x\in K: x(\xi)<0\}$ and
$V_\xi=\{x\in K: x(\xi)>0\}$. Note that $x_{\xi, 1}(\eta)=x_{\xi, -1}(\eta)=f_\eta(r_\xi)$
for any $\xi\not=\eta<\kappa$, and so $(x_{\xi,-1}, x_{\xi, 1})\in U_\eta\times V_\eta$
if and only if $\xi=\eta$ meaning that $(x_{\xi,-1}, x_{\xi, 1})_{\xi<\kappa}$ is discrete.

(5) Define $g_\xi=\overline{f_\xi}/2$ and $\mu_\xi=\delta_{x_{\xi,1}}-\delta_{x_{\xi, -1}}$.
We have 
\[\int g_\xi d\mu_\xi={1\over 2}(\overline{f_\xi}({x_{\xi,1}})-\overline{f_\xi}({x_{\xi, -1}}))=
{1\over 2}(1-(-1))=1\]
and 
\[\int g_\xi d\mu_\eta={1\over 2}(\overline{f_\xi}({x_{\eta,1}})-\overline{f_\xi}({x_{\eta, -1}}))=
{1\over 2}(f_\xi(r_\eta)-f_\xi(r_\eta))=0\]
if $\xi\not=\eta<\kappa$.

\end{proof}

Connectedly split intervals may or may not be hereditarily separable or hereditarily Lindel\"of.
This depends on the sequence of $f_\xi$s. The split intervals of Section 3
will be hereditarily Lindel\"of and hereditarily separable.

By a piecewise  polynomial  with rational
coefficients we will mean a function (not necessarily continuous)  whose domain $[0,1]$
can be divided into finitely many intervals on which the function is equal to a
polynomial with rational coefficients. By a rational polynomial we mean
a polynomial with rational coefficients which assumes all its local minima and maxima
in rational points. It is clear that rational polynomials on a bounded interval form a  dense set
of continuous functions of this interval, this could be observed by writing the derivative
of an arbitrary polynomial as the product of linear and indecomposable quadratic factors
and adjusting slightly the roots of the linear factors so that the integral of the derivative is
not affected much.

\begin{lemma}\label{representation} Suppose that $\{r_\xi:\xi<\kappa\}\subseteq [0,1]$ where $\kappa$
is an uncountable cardinal. Suppose that 
$K\subseteq [0,1]^{\{\st\}}\times [-1,1]^\kappa$ is a connectedly split interval induced by 
$(f_\xi)_{\xi<\kappa}$ where $f_\xi: [0,1]\setminus\{r_\xi\}\rightarrow [-1,1]$.
 Let $\varepsilon>0$ and $f\in C(K)$, then there is a (possibly
discontinuous) function
$g:K\rightarrow \R$ of the form
\[ g=P_{-1}\circ \overline{f_{\st}}+\Sigma_{i\leq k} P_i(\overline{f_{\xi_i}})\chi_{(s_i, s_i')_K}\]
where $k\in \N$,  $\xi_1, ..., \xi_k<\kappa$, $P_{-1}$ is a piecewise  polynomial in one variable with rational
coefficients, $P_i$ for $1\leq i\leq k$ are rational polynomials in one variable,
$ s_i, s_i'$ are rationals such that
$r_{\xi_i}\in (s_i, s_i')$ and the intervals 
$(s_i, s_i')$s for ${1\leq i\leq k}$ are pairwise disjoint,  and
\[\|f-g\|_\infty<\varepsilon.\]
\end{lemma}
\begin{proof} Recall from Definition \ref{splitdefinition}
that  $\overline{f_{\xi}}(x)=x(\xi)$ for $\xi<\kappa$
and $\overline{f_{\st}}(x)=x(\st)$ for all $x\in K$.
The coordinates of $K\subseteq [0,1]^{\{\st\}}\times[-1,1]^\kappa$ separate the points
of $K$, so by the Stone-Weierstrass theorem every function in $C(K)$ can be
approximated by  a  function from the algebra generated by 
the constant functions and $\overline{f_\xi}$s for $\xi<\kappa$ and by $\overline{f_{\st}}$.
In this algebra  the functions of the form
\[ Q(\overline{f_{\st}}, \overline{f_{\xi_1}}, ..., \overline{f_{\xi_k}})\]
are dense where $Q$ is a  polynomial in $(k+1)$-variables with rational coefficients
and $\xi_1<  ...<\xi_n$ for some $k\in \N$.
So we may assume that $f=Q(\overline{f_{\st}}, \overline{f_{\xi_1}}, ..., \overline{f_{\xi_k}})$
for one of such polynomials.

Note that  for $\xi<\kappa$ we have 
 $\overline {f_{\xi}}(x)=f_\xi(x(\st))=f_\xi(\overline{f_{\st}}(x))$ for $x\in K$ outside of the intervals
$(a, b)_K$ for $0\leq a<r_\xi<b\leq 1$ and that $f_\xi:[0,1]\setminus\{r_\xi\}\rightarrow [-1,1]$
 is continuous outside  $(a,b)$ for such $a,b$.
It follows that around $r_{\xi_i}$ there are intervals $(s_i, s_i')\ni r_{\xi_i}$ with rational endpoints
where the oscillations of $f_{\xi_j}$ for $j\not=i$ and $1\leq j\leq k$ 
is arbitrarily small. So for each $1\leq i\leq k$ find such an interval $(s_i, s_i')\ni r_{\xi_i}$
that 
\[ |Q(\overline f_{\st}, \overline f_{\xi_1}, ..., \overline f_{x_{\xi_i}}, ...,  \overline f_{\xi_k})(x)
-Q(q_{i, \st}, q_{i, 1}, ..., \overline f_{x_{\xi_i}}, ...,   q_{i, k})(x)|< \varepsilon\]
for some rationals $\{q_{i, \st}, q_{i, j}: -1\leq j\leq k, \  j\not=i\}$ and all $x\in  (s_i, s_i')_K$.
Let $P_i(v)=Q(q_{i, \st}, q_{i, 1}, ..., v, ...,   q_{i, n})$ be the corresponding polynomial in one variable
$v$, that is for all $x\in  (s_i, s_i')_K$ we have
\[ |f(x)-P_i(\overline f_{x_{\xi_i}})(x)|< \varepsilon.\]
We may assume it is a rational polynomial by the density of such polynomials.
Outside of the intervals $(s_i, s_i')_K$ the function $f=Q$ depends only on the
$\st$-coordinate and so it can be approximated by
a polynomial in the $\st$-coordinate with rational coefficients. Joining together
these polynomials with the zero function on the intervals $(s_i, s_i')$
for $1\leq i\leq k$ we obtain a piecewise polynomial $P_{-1}$ in one variable with rational coefficients
 such that 
\[ |f(p)-P_{-1}\circ \overline{f_{\st}}(p)|< \varepsilon\]
for all $p\in K\setminus \bigcup_{1\leq i\leq n}(s_i, s_i')_K$. This completes the proof.
\end{proof}

\section{An anti-Ramsey sequence of splitting functions}

\begin{definition}\label{randomdef} Suppose that $\{r_\xi:\xi<\kappa\}\subseteq [0,1]$ where $\kappa$
is an uncountable cardinal. Suppose that  $f_\xi: [0,1]\setminus\{r_\xi\}\rightarrow [0,1]$
are continuous in their domains. We say that the sequence $(f_\xi)_{\xi<\kappa}$ is anti-Ramsey if given:
\begin{enumerate}[(a)]
\item $m\in \N$,
\item pairwise disjoint nonempty open intervals $I_1, ..., I_m$ of $[0,1]$ with rational endpoints,
\item any uncountable
sequence $(F_\alpha)_{\alpha<\omega_1}$ where $F_\alpha=\{\xi_1^\alpha, ..., \xi_m^\alpha\}$
are pairwise disjoint finite subsets of $\kappa$ such that
$r_{\xi_i^\alpha}\in I_i$ for every $1\leq i\leq m $ and every $\alpha<\omega_1$,
\item  any two $m$-tuples $\{q_1, ..., q_m\}$ and $\{q_1', ..., q_m'\}$ of rationals from $[-1,1]$
\end{enumerate}
there are $\alpha<\beta<\omega_1$ and  open
intervals $(J_i^\alpha)_{1\leq i\leq m}$, and $(J_i^\beta)_{1\leq i\leq m}$ such that
 for each $1\leq i\leq m$ we have:
\begin{enumerate}
\item $J_i^\alpha,  J_i^\beta\subseteq I_i$; $J_i^\alpha\cap J_i^\beta=\emptyset$,
\item $r_{\xi_i^\alpha}\in J_i^\alpha$ and $r_{\xi_i^\beta}\in J_i^\beta$,
\item $f_{{\xi_i^{\alpha}}}\restriction ([0,1]\setminus (J_i^\alpha\cup  J_i^\beta))
=f_{{\xi_i^{\beta}}}\restriction ([0,1]\setminus (J_i^\alpha\cup  J_i^\beta))$,
\item $f_{{\xi_i^{\alpha}}}\restriction J_{i}^\beta=q_i$,
\item $f_{{\xi_i^{\beta}}}\restriction J_{i}^\alpha=q_i'$.
\end{enumerate}
\end{definition}

 \begin{proposition}\label{nonequirandomprop} 
Suppose that $\{r_\xi:\xi<\kappa\}\subseteq [0,1]$ where $\kappa$
is an uncountable cardinal. Suppose that 
$K\subseteq [0,1]^{\{\st\}}\times[-1,1]^\kappa$ is  connectedly split interval induced by 
 an anti-Ramsey sequence 
of functions $(f_\xi)_{\xi<\kappa}$ where $f_\xi: [0,1]\setminus\{r_\xi\}\rightarrow [0,1]$.
Suppose that $(g_\alpha)_{\alpha<\omega_1}$ is
is a sequence of  functions in $C(K)$ of the form
\[g_\alpha=\Sigma_{i\leq m} P_i(\overline{f_{\xi_i^\alpha}})\chi_{{(s_i, s_i')}_K}\]
where
\begin{itemize}
\item  $F_\alpha=\{\xi_1^\alpha, ..., \xi_m^\alpha\}$ are pairwise disjoint 
finite subsets of $\omega_1$ for $\alpha<\omega_1$,
\item $P_i$s are  rational polynomials  and  $P_i[[-1,1]]=[a_i, b_i]$ for 
some $a_i, b_i\in \R$
\item $\max\{|b_i-a_i|: 1\leq i\leq m\}=\delta>0$,
\item   $s_i, s_i'$ are rationals 
such that $r_{\xi_i^\alpha}\in (s_i, s_i')$ and
$(s_i, s_i')$s for  ${1\leq i\leq m}$ are pairwise disjoint
\end{itemize}
Then there are $\alpha<\beta<\omega_1$ such that $\|g_\alpha-g_\beta\|_\infty=\delta/2$
and there  are $\alpha'<\beta'<\omega_1$ such that $\|g_{\alpha'}-g_{\beta'}\|_\infty=\delta$.
\end{proposition}
\begin{proof} 
To find $\alpha<\beta<\omega_1$ such that $\|g_\alpha-g_\beta\|=\delta/2$
for every $1\leq i\leq m$ find $q_i=q_i'\in [-1,1]$ such that $P_i(q_i)$ is the middle
point of $[a_i, b_i]$ and apply Definition \ref{randomdef} for $I_i=(s_i, s_i')$. 
As $P_i$s are rational polynomials, i.e., with rational coefficients and 
with local maxima and minima in rational points, $q_i$s are rational as required in \ref{randomdef}.
Here we also use the connectedness of $[-1,1]$, that is, this point would not go
through if we had the Cantor set as in the papers \cite{rolewicz, biorthogonal} instead of $[-1,1]$.
By (3) of Definition \ref{randomdef} the functions  $g_\alpha$ and $g_\beta$  may only differ
on $J_i^\alpha$ and $J_i^\beta$ for $1\leq i\leq m$, actually
\[\|g_\alpha-g_\beta\|_\infty=
\max_{1\leq i\leq m}\big(\max_{x\in (J_i^\alpha)_K} |P_i(\overline{f_{\xi_i^\alpha}})(x))-P_i(q_i)|,
\max_{x\in (J_i^\beta)_K} |P_i(\overline{f_{\xi_i^\beta}}(x))-P(q_i'))|\big)\]
As $\overline{f_{\xi_i^\alpha}}$ and $\overline{f_{\xi_i^\beta}}$ have the range $[-1,1]$  and assume
 all these values in $R_{\xi_i^\alpha}
\subseteq (J_i^{\alpha})_K$ and $R_{\xi_i^\beta}
\subseteq (J_i^{\delta})_K$ respectively we conclude that $\|g_\alpha-g_\beta\|=\delta/2$
by the choice of $q_i=q_i'$ for all $1\leq i\leq k$.

To find $\alpha'<\beta'<\omega_1$ such that $\|g_{\alpha'}-g_{\beta'}\|=\delta$
for every $1\leq i\leq k$ find $q_i=q_i'\in [-1,1]$ such that $P_i(q_i)$ is the minimum
point of $[a_i, b_i]$ that is $P_i(q_i)=a_i$ and apply Definition \ref{randomdef} for $I_i=(s_i, s_i')$. 
By (3) of Definition \ref{randomdef} $g_\alpha$ and $g_\beta$  may only differ
on $J_i^\alpha$ and $J_i^\beta$ for $1\leq i\leq m$, actually
\[\|g_\alpha-g_\beta\|_\infty=\max_{1\leq i\leq m}\big(\max_{x\in (J_i^\alpha)_K} 
|P_i(f_{\xi_i^\alpha})(x))-P_i(q_i)|,
\max_{x\in (J_i^\beta)_K} |P_i(f_{\xi_i^\beta}(x))-P_i(q_i'))|\big)\]
As $f_{\xi_i^\alpha}$ and $f_{\xi_i^\beta}$ have the range $[-1,1]$  and assume
 all these values in $R_{\xi_i^\alpha}
\subseteq (J_i^{\alpha})_K$ and $R_{\xi_i^\beta}
\subseteq (J_i^{\delta})_K$ respectively we conclude that $\|g_\alpha-g_\beta\|=
\max_{1\leq i\leq k}|b_i-a_i|=\delta$
by considering $x\in R_{\xi_i^\alpha}$ where $P_i(f_{\xi_i^\alpha})(x))=b_i$
by the choice of $q_i$ for all $1\leq i\leq k$.
\end{proof}

\begin{theorem}\label{maintheorem} Suppose that $\{r_\xi:\xi<\omega_1\}\subseteq [0,1]$  Suppose that 
$K\subseteq [0,1]^{\{\st\}}\times[-1,1]^\kappa$ is a connectedly split interval induced by 
 an anti-Ramsey sequence
of functions $(f_\xi)_{\xi<\omega_1}$ where $f_\xi: [0,1]\setminus\{r_\xi\}\rightarrow [-1,1]$.
Then
$C(K)$ has no uncountable equilateral set nor an uncountable $(1+\varepsilon)$-separated
set in the unit ball.
\end{theorem}
\begin{proof} Let $\{e_\alpha:\alpha<\omega_1\}$ be  a collection of distinct elements of
the unit ball of $C(K)$. Let $\delta\in \R$ be the maximal real such that
for every $\alpha<\omega_1$, for every $\varepsilon>0$ there are
 $\alpha<\beta,\xi<\omega_1$
such that $diam(e_\beta[R_\xi])>\delta-\varepsilon$.  As the $e_\alpha$s are from 
the unit ball it is clear that $\delta\leq 2$. 

Now let us note 
we may assume that $\delta>0$: otherwise
there is $\alpha<\omega_1$ such that for every  $\alpha<\beta,\xi<\omega_1$
$diam(e_\beta[R_\xi])=0$; but two points $x, y\in K$ of a connectedly split interval
which first differ at a coordinate above $\alpha$ must both belong to some $R_\xi$
for $\xi>\alpha$. This means that for $\beta>\alpha$ the functions
$e_\beta$ depend only on the countably many coordinates from
 $\{\st\}\cup[0,\alpha]$,
but $C([0,1]\times[-1,1]^{\alpha+1})$ is a separable Banach space so 
among $e_\alpha$s for $\alpha<\omega_1$
 there are functions arbitrarily close to each other and so they cannot form
an equilateral set. Thus we may assume that $\delta>0$ indeed.

By Lemma \ref{representation} we can approximate each $e_\alpha$  by a 
function $g_\alpha'$ of the form specified in that lemma with $\|e_\alpha-g_\alpha'\|_\infty<{\delta\over 20}$.
Choosing $\varepsilon={\delta\over 20}$ above we conclude that for arbitrarily high
$\beta,\xi<\omega_1$ we have $diam(g_\beta'[R_\xi])\in ({19\over 20}\delta, {21\over 20}\delta)$.
Using the $\Delta$-system lemma and the fact that there are only countably many polynomials 
with rational coefficients and countably many intervals with rational endpoints
and using the maximality
of
$\delta$ by going to a subsequence we may assume that for every $\alpha<\omega_1$
\[g_\alpha'=P_{-1}\circ\overline{f_{*}}+\Sigma_{1\leq i\leq m}
 P_i(\overline{f_{\xi_i^\alpha}})\chi_{(s_i, s_i')}+ \Sigma_{m< i\leq k}
 P_i(\overline{f_{\xi_i}})\chi_{(s_i, s_i')}\]
where
\begin{itemize}
\item  $F_\alpha=\{\xi_1^\alpha, ..., \xi_m^\alpha\}$ are pairwise disjoint 
finite subsets of $\omega_1$ for $\alpha<\omega_1$ all disjoint from
$\{\xi_{m+1}, ..., \xi_k\}$,
\item $P_{-1}$ is a piecewise polynomial with rational coefficients,
\item $P_i$s for $1\leq i\leq k$ are rational polynomials,
\item   $s_i, s_i'$ are rationals 
such that $r_{\xi_i^\alpha}\in (s_i, s_i')$ for $1\leq i\leq m$ and
$r_{\xi_i}\in (s_i, s_i')$ for $m< i\leq k$ and
\item $(s_i, s_i')$s for  ${1\leq i\leq k}$ are pairwise disjoint, 
\item $P_i[[-1,1]]=[a_i, b_i]$ for some $a_i, b_i\in \R$ and all $1\leq i\leq k$,
\item $\max\{|b_i-a_i|: 1\leq i\leq m\}\in({19\over 20}\delta, {21\over 20}\delta)$.
\end{itemize}
The last item follows from the fact that $diam(g_\alpha'[R_\xi])\in({19\over 20}\delta, {21\over 20}\delta)$
for some $\xi<\omega_1$ and $\overline{f_\xi}$ separates points of $R_{\xi'}$
if and only if $\xi=\xi'$. For $\alpha<\omega_1$ put
\[g_\alpha=\Sigma_{1\leq i\leq m}
 P_i(\overline{f_{\xi_i^\alpha}})\chi_{(s_i, s_i')}\]
and note that $g_\alpha-g_\beta=g_\alpha'-g_\beta'$ for every $\alpha, \beta<\omega_1$.
Using Proposition \ref{nonequirandomprop}
we find  $\alpha<\beta<\omega_1$ such that $\|g_\alpha-g_\beta\|\leq {11\over 20}\delta$
and  $\alpha'<\beta'<\omega_1$ such that $\|g_{\alpha'}-g_{\beta'}\|\geq{18\over 20}\delta$.
So $\|e_\alpha-e_\beta\|\leq {13\over 20}\delta$ and $\|e_{\alpha}-e_{\beta'}\|\geq {16\over 20}\delta$
which shows that the functions $e_\alpha$s do not form an
equilateral set. The last part of the statement of the theorem follows from
Theorem 1 of \cite{equilateralcofk}.
\end{proof}

\section{Forcing}

In this section we use the method of forcing to show 
the consistency of the existence of an anti-Ramsey sequence
of functions $(f_\xi)_{\xi<\omega_1}$. 

\begin{definition}\label{piecewise} Let $0<a<b<1$. 
We say that $f:[0,a]\cup[b,1]\rightarrow [-1,1]$ is rationally piecewise linear if it is continuous on $[0,a]\cup[b,1]$ and 
there are pairwise disjoint intervals  $(I_i)_{i\leq k}$ with rational endpoints for $1\leq i\leq k\in \N$
such that $[0,a]\cup[b,1]=\bigcup_{1\leq i\leq k}I_i$ and  $f(r)=a_ir+b_i$ for all $r\in I_i$ and
some rationals $a_i, b_i$ and each $1\leq i\leq k$. 
\end{definition}

Rationally piecewise linear functions
 defined on sets $[0, r-\varepsilon]\cup[r+\varepsilon, 1]$ for some $\varepsilon>0$
and $r\in (0,1)$
will serve for us as approximations for splitting functions $f: [0,1]\setminus\{r\}\rightarrow [-1,1]$. 
We construct  a sequence of splitting functions
 in the generic extension of any set-theoretic universe obtained 
with the following forcing notion:

\begin{definition}\label{P} Fix an uncountable set $\{r_\xi:\xi<\omega_1\}$
of irrationals. $\PP$ is a forcing notion consisting of
triples $(n_p, F_p,  \mathcal F_p)$ such that:
\begin{enumerate}
\item $n_p\in \N$,
\item $F_p\subseteq \omega_1$ is a finite set,
such that there is an injective function sending $F_p\ni\xi\rightarrow k_\xi \in\{k\in \N: 0\leq k<2^{n_p}\}$
such that $r_\xi\in ({k_\xi\over2^{n_p}}, {k_\xi+1\over2^{n_p}})$,

\item $\mathcal F_p=\{f_p^\xi: \xi\in F_p\}$,
\item $f_p^\xi: [0,{k_\xi\over2^{n_p}}]\cup[{k_\xi+1\over2^{n_p}}, 1]\rightarrow [-1,1]$ is a
 rationally piecewise linear function
for each $0\leq k<2^{n_p}$.
\end{enumerate}
We say that $p\leq q$ if and only if 
\begin{enumerate}[(a)]
\item $n_p\geq n_q$,
\item $F_p\supseteq F_q$,
\item $f_p^\xi\supseteq f_q^\xi$ for every $\xi\in F_q$.
\end{enumerate}
\end{definition}

The motivation of the above definition should be clear. It is
the partial order of approximations to the anti-Ramsey sequence of functions as in Definition
\ref{randomdef}:
a given condition $p\in \PP$ provides approximations to
finitely many $f_\xi$s for $\xi\in F_p$ in the form  of a rationally piecewise linear function
$f_p^\xi: [0,{k_\xi\over2^{n_p}}]\cup[{k_\xi+1\over2^{n_p}}, 1]\rightarrow [-1,1]$.
The density argument in Proposition  \ref{forces} will allow us to prove that the
approximated sequence of splitting functions is anti-Ramsey.  First we need to take care
of the preservation of the cardinals.

\begin{lemma}\label{ccc} $\PP$ satisfies the c.c.c.
\end{lemma}
\begin{proof} Suppose that $(p_\alpha:\alpha<\omega_1)$ 
consists of distinct elements of $\PP$.  We will find $\alpha<\beta<\omega_1$
and $o\in \PP$ such that $o\leq p_\alpha, p_\beta$.
Using the fact
that there are countably many finite systems of
rationals and countably many rationally piecewise linear functions by going to an uncountble subsequence
we may assume that the following holds for all $\alpha<\omega_1:$
\begin{itemize}
\item $n_{p_\alpha}=n$ for some $n\in \N$,
\item $(F_{p_\alpha})_{\alpha<\omega_1}$ forms a $\Delta$-system $F_{p_\alpha}=\Delta\cup
\{\xi^\alpha_1,  ..., \xi^\alpha_m\}$
with root $\Delta=\{\xi_1, ..., \xi_l\}$, 
\item there are distinct $k_{1}', ..., k_{l}', k_1, ..., k_m\in \{k\in \N: 0\leq k<2^{n}\}$
such that 
\begin{itemize}
\item $r_{\xi_i}\in ({k_{i}'\over2^{n}}, {k_{i}'+1\over2^{n}})$
for $1\leq i\leq l$ and 
\item $r_{\xi_i^\alpha}\in ({k_{i}\over2^{n}}, {k_{i}+1\over2^{n}})$
for $1\leq i\leq m$;
\end{itemize}
\item $f_{p_\alpha}^{\xi_i}=f_{i}'$ for $1\leq i\leq l$ 
and  $f_{p_\alpha}^{\xi_i^\alpha}=f_i$ for $1\leq i\leq m$ are
 rationally piecewise linear functions such that
\begin{itemize}
\item $f_{i}': [0,{k_{i}'\over2^{n}}]\cup[{k_{i}'+1\over2^{n}}, 1]\rightarrow [-1,1]$,
\item $f_i: [0,{k_i\over2^{n}}]\cup[{k_i\over2^{n}}, 1]\rightarrow [-1,1].$
\end{itemize}
\end{itemize}
Fix $\alpha<\beta<\omega_1$ as above. We will construct $o\leq p, q$ with $o\in \PP$ where
$p=p_\alpha$ and $q=p_\beta$. We will have $F_o=F_{p}\cup F_q$. 
Let $n_o>n$ be such a positive integer that there is an injective
 function sending $F_o\ni\xi\rightarrow k_{\xi}'' \in\{k\in \N: 0\leq k<2^{n_o}\}$
such that $r_\xi\in ({k_{\xi}''\over2^{n_o}}, {k_{\xi}''+1\over2^{n_o}})$. 
It exists since all of our $r_\xi$s for $\xi<\omega_1$ are irrationals.

It is clear that we have
\begin{itemize}
\item $r_{\xi_i}\in ({k_{\xi_i}''\over2^{n_o}}, {k_{\xi_i}''+1\over2^{n_o}})\subseteq ({k_{i}'\over2^{n}}, {k_{i}'+1\over2^{n}})$
for $1\leq i\leq l$ and 
\item $r_{\xi_i^\alpha}\in ({k_{\xi_i^\alpha}''\over2^{n_o}}, {k_{\xi_i^\alpha}''+1\over2^{n_o}})
\subseteq ({k_{i}\over2^{n}}, {k_{i}+1\over2^{n}})$
for $1\leq i\leq m$ and 
\item $r_{\xi_i^\beta}\in ({k_{\xi_i^\beta}''\over2^{n_o}}, {k_{\xi_i^\beta}''+1\over2^{n_o}})
\subseteq ({k_{i}\over2^{n}}, {k_{i}+1\over2^{n}})$
for $1\leq i\leq m$.
\end{itemize}
So we are left with 
\begin{itemize}
\item extending each $f_p^{\xi_i^\alpha}$ for $\xi_i^\alpha \in F_p\setminus\Delta$ from
$ [0,{k_{i}\over2^{n}}]\cup[{k_{i}+1\over2^{n}}, 1]$ to a 
 rationally piecewise linear function defined on
$ [0,{k_{\xi_i^\alpha}''\over2^{n_o}}]\cup[{k_{\xi_i^\alpha}''+1\over2^{n_o}}, 1]$
for $1\leq i\leq m$,
\item extending  each $f_q^{\xi_i^\beta}$ for ${\xi^i_\beta} \in F_q\setminus \Delta$ from
$ [0,{k_{i}\over2^{n}}]\cup[{k_{i}+1\over2^{n}}, 1]$ to a 
 rationally piecewise linear function defined on
$ [0,{k_{\xi_i^\beta}''\over2^{n_o}}]\cup[{k_{\xi_i^\beta}''+1\over2^{n_o}}, 1]$
for $1\leq i\leq m$,
\item 
extending  each $f_q^{\xi_i}$ and $f_p^{\xi_i}$ for $\xi_i \in\Delta$  for $1\leq i\leq l$ from
$ [0,{k_{i}'\over2^{n}}]\cup[{k_{i}'+1\over2^{n}}, 1]$ to a single (for each $1\leq i\leq l$)
 rationally piecewise linear function defined on
$ [0,{k_{\xi_i}''\over2^{n_o}}]\cup[{k_{\xi_i}''+1\over2^{n_o}}, 1]$ for $1\leq i\leq l$.
\end{itemize}
Obviously in the first two cases there are some extensions as above.
In the third case we need to use the fact that 
for each $\xi_i\in \Delta$ both of the functions $f_q^{\xi_i}$ and $f_p^{\xi_i}$ are the same 
because they are equal
to $f_i'$ for $1\leq i\leq l$.
This completes the construction
of $o\leq p, q$ and the proof of the c.c.c. for $\PP$.
\end{proof}

\begin{proposition}\label{forces} $\PP$ forces that
there is an anti-Ramsey sequence $(f_\xi)_{\xi<\check{\omega}_1}$.
\end{proposition}
\begin{proof}
In the ground set-theoretic universe we fix a sequence $\{r_\xi:\xi<\omega_1\}\subseteq [0,1]$ 
of irrationals and force over this universe with $\PP$ as in Definition \ref{P}.
In the generic extension for each $\alpha<{\check\omega}_1$ we choose
finite pairwise disjoint $F_\alpha$s 
and rationals $q_i, q_i'$ for $1\leq i\leq m\in \N$ and intervals $I_i$s as in Definition \ref{randomdef}.

Let $\dot G$ be the standard $\PP$-name for the generic filter in $\PP$ (see \cite{kunen}).
In the generic extension define   functions
$f_\xi: [0,1]\setminus\{r_\xi\}\rightarrow [-1,1]$ as the unique continuous extension
 of $\bigcup\{f_p^\xi:  p\in G\}$ to the new reals in $[0,1]\setminus\{r_\xi\}$.
This can be done in a standard way since $f_p^\xi$s are  uniformly continuous
in closed intervals away from $r_\xi$ respectively. The domain of $f_\xi$ is indeed $[0,1]\setminus\{r_\xi\}$
as using the technique of Lemma \ref{ccc} one can check that
\[\{p\in \PP: n_p>n\}\leqno (*)\]
is dense in $\PP$ for every $n\in \N$. 
From now on  work in the ground universe. Let $\dot f_\xi$ be a $\PP$-name for $f_\xi$.
It is clear that each $p\in \PP$ forces that $f_p^\xi\subseteq \dot f_\xi$. This will
be sufficient for establishing the proposition.
Fix $p\in \PP$  and names $\dot F_\alpha$ such that $p$ forces that 
$\dot F_\alpha$s are pairwise  disjoint $m$-element subsets of $\omega_1$.
By going to a stronger condition we may assume that $p$ decides the endpoints
of all intervals $I_i$ for $1\leq i\leq m$. Let $d\in \Q$ be the diameter of the smallest of them.
Let $p_\alpha\leq p$ and $F_\alpha=\{\xi_1^\alpha, ..., \xi_m^\alpha\}$ be 
in the ground universe such that
$p_\alpha\forces \dot{F}_\alpha=\check{F}_\alpha$. We may assume that $F_\alpha$s
form a $\Delta$-system, but since we have the c.c.c. and $p$ forces that they are
pairwise disjoint the root of the $\Delta$-system must be empty, that is
$F_\alpha$s are pairwise disjoint. Noting that the sets
\[\{p\in \PP: F\subseteq F_p\}\]
are dense in $\PP$ for every finite $F\subseteq \omega_1$ by extending
$p_\alpha$s we may assume that
$F_\alpha\subseteq F_{p_\alpha}$.  As for every $\alpha<\omega_1$ 
the condition $p_\alpha$ forces that 
$r_{\xi_i^\alpha}\in I_i$ we may assume that there is a fixed $n'\in \N$
such that $(r_{\xi_i^\alpha}-{1\over 2^{n'}}, r_{\xi_i^\alpha}+{1\over 2^{n'}})\subseteq I_i$
for every $1\leq i\leq m$.
 By going to an uncountable subsequence
we may assume that the following holds for all $\alpha<\omega_1:$
\begin{itemize}
\item $n_{p_\alpha}=n$ for some $n\in \N$  with $n\geq n'$,
\item $(F_{p_\alpha})_{\alpha<\omega_1}$ forms a $\Delta$-system $F_{p_\alpha}=\Delta\cup
\{\xi^\alpha_1,  ...,\xi^\alpha_{m}, \xi^\alpha_{m+1}...,  \xi^\alpha_{m'}\}$
with root $\Delta=\{\xi_1, ..., \xi_l\}$, 
\item there are distinct $k_{1}', ..., k_{l}', k_1, ..., k_{m'}\in \{k\in \N: 0\leq k<2^{n}\}$
such that 
\begin{itemize}
\item $r_{\xi_i}\in ({k_{i}'\over2^{n}}, {k_{i}'+1\over2^{n}})$
for $1\leq i\leq l$ and 
\item $r_{\xi_i^\alpha}\in ({k_{i}\over2^{n}}, {k_{i}+1\over2^{n}})$
for $1\leq i\leq m'$ and
\item $ ({k_{i}\over2^{n}}, {k_{i}+1\over2^{n}})\subseteq I_i$
for $1\leq i\leq m$ and
\end{itemize}
\item $f_{p_\alpha}^{\xi_i}=f_{i}'$ for $1\leq i\leq l$ 
and  $f_{p_\alpha}^{\xi_i^\alpha}=f_i$ for $1\leq i\leq m'$ and some
 rationally piecewise linear functions such that
\begin{itemize}
\item $f_{i}': [0,{k_{i}'\over2^{n}}]\cup[{k_{i}'+1\over2^{n}}, 1]\rightarrow [-1,1]$
\item $f_i: [0,{k_i\over2^{n}}]\cup[{k_i+1\over2^{n}}, 1]\rightarrow [-1,1]$
\end{itemize}
\end{itemize}
Fix $\alpha<\beta<\omega_1$ as above. As in the proof of the c.c.c. of $\PP$
we will construct $o\leq p, q$ where
$p=p_\alpha$ and $q=p_\beta$ such that $o$ will force the
properties $f_{\xi_\alpha^i}$ and $f_{\xi_\beta^i}$
 as in Definition \ref{randomdef}. We will have $F_o=F_{p}\cup F_q$. 
Let $n_o>n$ be such a positive integer that there is an injective
 function sending $F_o\ni\xi\rightarrow k_{\xi}'' \in\{k\in \N: 0\leq k<2^{n_r}\}$
such that $r_\xi\in ({k_{\xi}''\over2^{n_o}}, {k_{\xi}''+1\over2^{n_o}})$. 
By  choosing sufficiently big $n_o$
we may assume that 
\[ |r_{\xi_i^\alpha}-{k_{i}\over2^{n}}|, |r_{\xi_i^\alpha}-{k_{i}+1\over2^{n}}|> 
{1\over2^{n_o}}\leqno \bigodot\]
holds for every $1\leq i\leq m$.
It is clear that we have
\begin{itemize}
\item $r_{\xi_i}\in ({k_{\xi_i}''\over2^{n_o}}, {k_{\xi_i}''+1\over2^{n_o}})\subseteq ({k_{i}'\over2^{n}}, {k_{i}'+1\over2^{n}})$
for $1\leq i\leq l$ and 
\item $r_{\xi_i^\alpha}\in ({k_{\xi_i^\alpha}''\over2^{n_o}}, {k_{\xi_i^\alpha}''+1\over2^{n_o}})
\subseteq ({k_{i}\over2^{n}}, {k_{i}+1\over2^{n}})$
for $1\leq i\leq m'$ and 
\item $r_{\xi_i^\beta}\in ({k_{\xi_i^\beta}''\over2^{n_o}}, {k_{\xi_i^\beta}''+1\over2^{n_o}})
\subseteq ({k_{i}\over2^{n}}, {k_{i}+1\over2^{n}})$
for $1\leq i\leq m'$.
\end{itemize}
We define
\begin{itemize}
\item  $J_i^\alpha= ({k_{\xi_i^\alpha}''\over2^{n_o}}, {k_{\xi_i^\alpha}''+1\over2^{n_o}})$ 
for $1\leq i\leq m$,
\item $J_i^\beta=({k_{\xi_i^\beta}''\over2^{n_o}}, {k_{\xi_i^\beta}''+1\over2^{n_o}})$
for  $1\leq i\leq m$.
\end{itemize}
So we have $J_i^\alpha, J_i^\beta\subseteq I_i$ for each $1\leq i\leq m$
but by $\bigodot$ the endpoints of $J_i$s do not coincide with
the endpoints of the intervals $[{k\over 2^n}, {k+1\over 2^n}]$ for
$0\leq k<2^n$. We are left with 
\begin{itemize}
\item extending each $f_p^{\xi_i^\alpha}$ for $\xi_i^\alpha \in F_p\setminus\Delta$ from
$ [0,{k_{i}\over2^{n}}]\cup[{k_{i}+1\over2^{n}}, 1]$ to a 
 rationally piecewise linear function $f_o^{\xi_i^\alpha}$ defined on
\[[0,{k_{\xi_i^\alpha}''\over2^{n_o}}]\cup[{k_{\xi_i^\alpha}''+1\over2^{n_o}}, 1]=[0,1]\setminus J_i^\alpha,\]
and extending  each $f_q^{\xi_i^\beta}$ for $\xi_i^\beta \in F_q\setminus \Delta$ from
$ [0,{k_{i}\over2^{n}}]\cup[{k_{i}+1\over2^{n}}, 1]$ to a 
 rationally piecewise linear function $f_o^{\xi_i^\beta}$ defined on
\[[0,{k_{\xi_i^\beta}''\over2^{n_o}}]\cup[{k_{\xi_i^\beta}''+1\over2^{n_o}}, 1]=[0,1]\setminus J_i^\beta,\] so that
\begin{enumerate}
\item $f_o^{{\xi_i^{\alpha}}}\restriction ([0,1]\setminus (J_i^\alpha\cup  J_i^\beta))
=f_o^{{\xi_i^{\beta}}}\restriction ([0,1]\setminus (J_i^\alpha\cup  J_i^\beta))$,
\item $f_o^{{\xi_i^{\alpha}}}\restriction J_{i}^\beta=q_i$,
\item $f_o^{{\xi_i^{\beta}}}\restriction J_{i}^\alpha=q_i'$.
\end{enumerate}
\item 
extending  each $f_q^\xi$ and $f_p^\xi$ for $\xi \in\Delta$ from
$ [0,{k_{i}'\over2^{n}}]\cup[{k_{i}'+1\over2^{n}}, 1]$ to a single for each $\xi\in \Delta$
 rationally piecewise linear function defined on
$ [0,{k_{\xi}''\over2^{n_o}}]\cup[{k_{\xi}''+1\over2^{n_o}}, 1]$,
\end{itemize}
As we noted in the proof of the c.c.c. of $\PP$ any extensions will do to obtain 
$o\leq p,q$, however
to obtain the properties of $f_{\xi_\alpha^i}$ and $f_{\xi^\beta_i}$ from Definition \ref{randomdef}
one needs to follow the requirements (3) - (5) of it that is obtain (1)-(3) above.
This can be achieved as $f_{p}^{\xi_i^\alpha}=f_{q}^{\xi_i^\beta}=f_i$
which is defined on $[0,1]\setminus ({k_i\over{2^n}}, {k_i+1\over{2^n}})$
and $J_i^\alpha, J_i^\beta\subseteq [{k_i\over{2^n}}, {k_i+1\over{ 2^n}}]$ are
 disjoint for each $1\leq i\leq m$ and have the endpoints distinct from 
${k_i\over2^n}, {k_i+1\over 2^n}$ by $\bigodot$.

In the third case we need to use the fact that 
for each $\xi_i\in \Delta$ both of the functions $f_q^{\xi_i}$ and $f_p^{\xi_i}$ are the same 
because they are equal
to $f_i'$ for $1\leq i\leq l$.

This completes the construction
of $o\leq p, q$ which forces the properties $f_{\xi^\alpha_i}$ and $f_{\xi^\beta_i}$ from Definition \ref{randomdef} .
\end{proof}

\begin{theorem}\label{consistency} It is relatively consistent with ZFC that there is
a sequence $(f_\xi)_{\xi<\omega_1}$ which is anti-Ramsey.
\end{theorem}
\begin{proof} As $\PP$ is c.c.c. by Lemma \ref{ccc} the $\omega_1$ of the 
generic extension is the same as $\omega_1$ of the ground universe, and so 
Proposition \ref{forces} implies that there is an anti-Ramsey sequence 
in the generic extension.
\end{proof}

\section{Equilateral sets in Banach spaces of the form $C(K)$ under Martin's axiom}

\begin{theorem}
Assume Martin's axiom and the negation of the continuum hypothesis. Suppose that
$K$ is a compact Hausdorff space.
If the Banach space $C(K)$ of real valued continuous functions on $K$ wit the supremum norm
 is nonseparable, then the unit ball of $C(K)$ contains
\begin{itemize}
\item an uncountable $2$-equilateral set,
\item an uncountable $(1+\varepsilon)$-separated set.
\end{itemize}

\end{theorem}
\begin{proof}
First let us note that we may assume that the weight of $K$, that is the
density of $C(K)$ is the first uncountable cardinal $\omega_1$. Indeed, 
by embedding $K$ into $[0,1]^\kappa$ for some uncountable $\kappa$ one may find 
an uncountable $A\subseteq \kappa$ such that its cardinality is $\omega_1$ and 
the projection from $[0,1]^\kappa$ onto $[0,1]^A$ sends $K$ to a set of uncountable
weight. So we obtain a continuous image $L$ of $K$ which is of weight $\omega_1$,
it is clear that $C(L)$ has its isometric copy inside $C(K)$ and so
an equilateral set in $C(L)$ yields an equilateral set in $C(K)$.
We may also assume that $K$ is hereditarily Lindel\"of by Theorems 1 and  2 (2)
of \cite{equilateralcofk}.  We will need the following fact about hereditarily Lindel\"of
compact Hausdorff spaces:
\vskip 6pt
\noindent{\bf Claim:} If $K$ is compact Hausdorff hereditarily Lindel\"of, $U, V$ are disjoint
open subsets of $K$, then there is a continuous function
$f: K\rightarrow[-1,1]$ such that $f(x)>0$ for all $x\in U$ and $f(x)<0$ 
for all $x\in V$.\par

\noindent{\it Proof of the Claim:}
By the regularity of $K$ both $U$ and $V$ can be covered by
their open subsets whose closures are included in $U$ and $V$ respectively.
As $K$ is hereditarily Lindel\"of, we can find countable subcovers of these covers
which yield sequences of open sets $(U_n)_{n\in \N}$ and $(V_n)_{n\in \N}$
such that $\bigcup_{n\in \N}U_n=U$, $\bigcup_{n\in \N}V_n=V$ and
$\overline{U_n}\subseteq U$ as well as $\overline{V_n}\subseteq V$ hold for each $n\in \N$.
It follows that $\overline{U_n}\cap\overline{V}=\emptyset$ as well as
$\overline{V_n}\cap\overline{U}=\emptyset$ for each $n\in \N$. Using the
normality of $K$ define
$f_n: K\rightarrow [0, {1\over{2^n}}]$ such that $f_n\restriction \overline{U_n}={1\over{2^n}}$ and
$f_n\restriction\overline{V}=0$ and likewise 
$g_n: K\rightarrow [{-1\over{2^n}}, 0]$ such that $g_n\restriction \overline{V_n}={-1\over{2^n}}$ and
$g_n\restriction\overline{U}=0$. Now  $f=\Sigma_{n\in \N}(f_n+g_n)$ works
as $\bigcup_{n\in N}\overline{U_n}=U$ and $\bigcup_{n\in N}\overline{V_n}=V$ which completes the proof of the claim.
\vskip 6pt
Let $\mathcal B$ be
an open basis for $K$ of size $\omega_1$.
Consider the partial ordering $\leq $ on the set
\[\PP=\{p=(U_p, V_p): \overline{U_p}\cap \overline{V_p}=\emptyset;\
 U_p, V_p\in\mathcal B\}\]
where $p\leq q$ if and only if $U_p\supseteq U_q$ and $V_p\supseteq V_q$.
Let us assume that $\PP$ is c.c.c. and we will derive a 
contradiction with the nonmetrizability of $K$. By
 the negation of the continuum hypothesis $|\PP|<2^\omega$ and so 
by Martin's axiom 
and
Theorem 4.5 from \cite{weiss}
$\PP=\bigcup_{n\in \N}\PP_n$ where each $\PP_n$ consists of compatible elements.
That is, for each $n\in \N$ and each $p, q\in \PP_n$ there is $r\in \PP$ such that $r\leq p,q$
which in particular means that $\overline{(U_p\cup U_q)}\cap \overline{(V_p\cup V_q)}=\emptyset$.
Consider $\bigcup\{U_p: p\in \PP_n\}=U_n$ and 
$\bigcup\{V_p: p\in \PP_n\}=V_n$
for each $n\in \N$ and note that the compatibility of the elements of $\PP_n$
gives that $U_n\cap V_n=\emptyset $ for each $n\in \N$.
Hence, following the claim, we can construct continuous $f_n: K\rightarrow [-1, 1]$ such that
$f\restriction U_n$ is strictly positive and $f\restriction V_n$ is strictly negative.
Now we note that given two distinct points $x, y\in K$ there is
$p\in \PP$ such that $x\in U_p$ and $y\in V_p$ as the open sets
are taken from a basis  and hence there is an $n\in \N$ such that $f_n$
 separates $x$ and $y$. The existence of a countable family
of continuous functions separating the points of $K$  contradicts the nonmetrizability of
$K$.

So we conclude that $\PP$ has an uncountable antichain $(p_\xi: \xi<\omega_1)$.
It means that $\overline{U_{p_\xi}}\cap \overline{V_{p_\eta}}\not=\emptyset$
or $\overline{V_{p_\xi}}\cap \overline{U_{p_\eta}}\not=\emptyset$ for
any distinct $\xi,\eta<\omega_1$ and in particular $({U_{p_\xi}},{V_{p_\eta}})\not=({U_{p_\eta}},{V_{p_\xi}})$,
hence we obtain an uncountable  intersecting family of closed pairs which is equivalent
by Theorem 1 of \cite{equilateralcofk} to the existence
of $2$-equilateral uncountable family in the sphere of $C(K)$ and an $(1+\varepsilon)$-separated
set for some $\varepsilon>0$.
\end{proof}

\bibliographystyle{amsplain}

\end{document}